\documentclass{amsart}
\usepackage{amsmath}
\usepackage{amsfonts}
\usepackage{verbatim}

\newtheorem{theorem}{Theorem}[section]
\theoremstyle{plain}

\newtheorem{corollary}[theorem]{Corollary}

\newtheorem{definition}[theorem]{Definition}

\newtheorem{lemma}[theorem]{Lemma}

\newtheorem{proposition}[theorem]{Proposition}
\theoremstyle{remark}
\newtheorem{remark}[theorem]{Remark}

\numberwithin{equation}{section}

\begin{document}
\title[Shrinking targets, fast mixing, geodesic flow]{Shrinking targets in
fast mixing flows and the geodesic flow on negatively curved manifolds}
\author{Stefano Galatolo $^1$}
\email{$^1$ s.galatolo@ing.unipi.it}
\address{Dipartimento di Matematica Applicata "U. Dini", Via \ Buonarroti 1,
Pisa}
\author{Isaia Nisoli $^2$}
\email{$^2$ nisoli@mail.dm.unipi.it}
\address{Dipartimento di Matematica "L. Tonelli", Largo Pontecorvo 5 , Pisa}

\begin{abstract}
We show that in a rapidly mixing flow with an invariant measure, the time
which is needed to hit a given section is related to a sort of conditional
dimension of the measure at the section. The result is applied to the
geodesic flow of compact variable negative sectional curvature manifolds,
establishing a logarithm law for such kind of flow.
\end{abstract}

\subjclass{37D40 ; 37A25; 37A10; 37J55}
\keywords{Flows; logarithm laws; shrinking targets; decay of correlations;
geodesic flow}
\maketitle

\section{Introduction}

Let $M$ be a differentiable manifold and let $\Phi ^{t}$ be a $C^{1}$ flow
in $M$. Let $A\subset M$ be a set. Let us consider the hitting time of the
orbit of $x\in M$ to the set $A$%
\begin{equation}
\tau (x,A)=\inf \{t\in \mathbb{R}^{+}:\Phi ^{t}(x)\in A\}.  \label{000}
\end{equation}

If we consider a decreasing sequence of shrinking target sets $A_{n}$,
sometimes it is worth to have an estimation about the time needed for
typical orbits to enter in $A_{n}$. If we suppose that the flow preserves a
measure $\mu $, in a large class of situations sharing fast mixing or
generic arithmetic behavior, a result of the following type holds%
\begin{equation}
\tau (x,A_{n})\sim \frac{1}{\mu (A_{n})}  \label{sim}
\end{equation}%
(where $\sim $ stand for some kind of more or less strict equivalence in the
asymptotic behavior, when $n\rightarrow \infty $). In discrete time systems
earlier results of this kind were given in \cite{Erdosren},\cite{Phi}.
Results of this kind are often called logarithm laws. As far as we know, the
first place where this name was used for such a result was in \cite{Su}.
Since it is quite strictly related to the main application we present here,
let us describe this classical result and relate it with the above formula.

Let $\mathbb{H}_{k+1}$ stand for the $k+1$-dimensional real hyperbolic space
(with curvature $-1$ ). Let us consider a discrete group $G$ of hyperbolic
isometries of $\mathbb{H}_{k+1}$ such that $Y=\mathbb{H}_{k+1}/G$ is not
compact and has finite volume. Let $T^{1}Y$ be its unitary tangent bundle.
Let $\pi_1 :T^{1}Y\rightarrow Y$ the canonical projection, $\Phi ^{t}$ the
geodesic flow on $T^{1}Y$, $\mu $ the Liouville measure on $T^{1}Y$, and $d$
the distance on $Y.$ In \cite{Su} Sullivan proved a law for the speed of
approaching of typical geodesics to the point at infinity.

\begin{theorem}[Sullivan]
For all $p\in Y$ and $\mu $ almost each $v\in T^{1}Y$%
\begin{equation}
\underset{t\rightarrow \infty }{\lim \sup }\frac{d(p,\pi_1 (\Phi ^{t}v))}{%
\log t}=\frac{1}{k}.  \label{loglaw1}
\end{equation}
\end{theorem}

It is interesting to remark that the above result implies that there are
infinitely many times $t_{1},t_{2},...$where the geodesic reaches a "locally
maximal" distance from $p$, and $t_{i}\sim e^{kd(p,\pi (\Phi ^{t_{i}}v))}$.
We remark that in the above example, setting the target set $Y_{L}=\{v\in
T^{1}Y,\,d(p,\pi (v))\geq L\}$ then $\mu (Y_{L})\sim e^{-kL}$ (see proof of
theorem 6 in \cite{Su}) thus it follows that in this example, to enter in
the set $Y_{L}$ you need a time of the order of $1/{\mu (Y_{L})}$ (to be
more precise, this holds for a sequence of values of $L$). This shows a
connection with formula \eqref{sim} and gives, as an example, the following
consequence%
\begin{equation}
\underset{L\rightarrow \infty }{\lim \inf }\frac{\log ~\tau (v,Y_{L})}{-\log
\mu (Y_{L})}=1.  \label{loglaw2}
\end{equation}

The Sullivan result on the geodesic flow was generalized to many other
contexts. For example, Kleinbock and Margulis (\cite{KM}) generalized it to
homogeneous spaces. Another generalization near to the context of the
present paper was given in \cite{Maucourant} (see Theorem \ref{mau}) where
this kind of result is extended to hyperbolic manifolds having constant
curvature and considering target sets which are balls of the base manifold%
\footnote{%
To be more precise, as we will see in the following, the target sets on the
space where the dynamics act (the tangent bundle to our manifold) are the
preimages of balls under the natural projection of the fibration.}.

Results similar to equation \eqref{loglaw1} using the Hausdorff dimension
instead of the invariant measure (a "Yarnik type" result)\ were given in 
\cite{BV} (see also \cite{HP2}). Logarithms laws for flows in other contexts
were given for example by \cite{AM}, \cite{GP} and \cite{Masur} (see also
the survey \cite{atr}).

In this paper we consider fastly mixing flows and give estimates for the
time which is needed to hit a given target, in the spirit of equation %
\eqref{loglaw2}. Equations like \eqref{loglaw2} then easily give logarithm
laws like \eqref{loglaw1}, and even more precise results (see Proposition %
\ref{loglaw5}). A main ingredient in this work is the speed of decay of
correlations.

It is nowadays well known that many aspects of the statistical behavior of a
deterministic chaotic system can be described by suitable versions of
theorems from the classical theory of probability (central limit, large
deviations, etc.); in the proof of those theorems, the independence
assumption, which is made in the probabilistic context, is often replaced by
weaker assumptions, also regarding the speed of correlations decay, which in
turn is related to the spectral properties of the system. This strongly
motivates the study of decay of correlations and spectral properties as a
tool to deduce many other consequences.

A general aim of the paper is to show how logarithm laws (and some more
precise shrinking target results) can follow directly from fast \emph{decay
of correlation} of the flow.

We will use decay of correlation with respect to Lipschitz observables. We
say that the system has decay of correlation with speed $\alpha $ with
respect to Lipschitz observables if for each $f,g:M\rightarrow \mathbb{R}$,
Lipschitz functions

\begin{equation}
\bigg|\int ~g(\phi ^{t}(x))f(x)d\mu -\int g(x)d\mu \int f(x)d\mu \bigg|\leq
C~\Vert g\Vert ~\Vert f||~\alpha (t)
\end{equation}%
where $\alpha (t)\rightarrow 0$ and $\left\vert \left\vert {.}\right\vert
\right\vert $ is the Lipschitz norm. We say that the speed of correlation
decay is superpolynomial if $\lim_{t\to\infty} t^{\gamma}\alpha(t)=0,$ $%
\forall \gamma >0.$

In Theorem \ref{GeneralResult} we will see that \emph{if the system has
superpolynomial speed of decay of correlations as above, the time which is
needed to cross a small transverse section is related to the measure of the
set of points which are sent to the section in a time less than $\epsilon ,$
in a sense similar to Eq. \eqref{loglaw2}. }

Moreover in Proposition \ref{propsec} we will see another general tool,
allowing to establish a logarithm law for a flow once we are able to obtain
it on a suitable Poincar\'{e} section, with the induced map.

As an application of Theorem \ref{GeneralResult} we show (see Theorem \ref%
{loglaw3} and the propositions that follow) a Maucourant type result (see
Theorem \ref{mau}) for (compact, connected) manifolds with \emph{variable}
negative sectional curvature. This will be a consequence of the above cited
Theorem \ref{GeneralResult} and of the exponential decay of correlation of
the geodesic flow of such manifolds, which was sharply estimated in \cite{Li}%
.

\noindent \textbf{Acknowledgements. }Isaia Nisoli would like to thank the
PRIN\ project "Metodi variazionali e topologici nello studio di fenomeni non
lineari" for funding is research through a contract of research with the
University of Bari.

\section{Logarithm law, superpolynomial decay, Poincar\'{e} sections}

Let $M$ be a differentiable manifold and let $\Phi ^{t}$ be a $C^{1}$ flow
in $M$. Let $A\subset M$ be a set. Let us consider the hitting time $\tau
(x,A)$ of the orbit of $x$ to the set $A$ as defined in Eq. \ref{000}.

We present the main general tool of the paper. A logarithm law is
established, in flows having fast enough decay of correlation.

\begin{theorem}
\label{GeneralResult} Suppose that $\Phi ^{t}$ as above preserves a measure $%
\mu $ and has superpolynomial decay of correlations with respect to
Lipschitz observables. Let $V$ a submanifold of $M$ which is transverse to $%
\Phi ^{t}$. Let $f:M\rightarrow \mathbb{R}^{+}$ be a Lipschitz function, let
us consider the target family%
\begin{equation}
B_{l}=\{x\in V,f(x)<l\}
\end{equation}%
and suppose that $B_{l}$ is bounded for some $l>0$. Let%
\begin{equation}
C_{\epsilon ,l}=\{x\in X,s.t.~\Phi ^{t}(x)\in B_{l},0\leq t<\epsilon \}.
\end{equation}%
Suppose that the following limit exists $\footnote{%
It is easy to see that since the measure is preserved, the outer limit ($%
\epsilon \rightarrow 0$) always exists.}$%
\begin{equation}
d=\lim_{\epsilon \rightarrow 0}\lim_{l\rightarrow 0}\frac{\log \mu
(C_{\epsilon ,l})}{\log l}
\end{equation}%
then%
\begin{equation}
\lim_{l\rightarrow 0}\frac{\log \tau (x,B_{l})}{-\log l}=d  \label{general}
\end{equation}%
for $\mu $-almost each $x.$
\end{theorem}

We remark that there are mixing, smooth flows for which logarithm laws and
similar results does not hold (see \cite{GP09}), hence in a result like
Theorem \ref{GeneralResult} some assumptions on the speed of decay of
correlations are important.

The proof is based on a result about discrete time systems we are going to
recall: let $(X,T,\mu )$ be an ergodic, measure preserving transformation on
a metric space $X$. Let us consider a \ family of target sets $S_{r}$
indexed by a real parameter $r$ and the time needed for the orbit of a point 
$x$ to enter in $S_{r}$%
\begin{equation*}
\tau _{T}(x,S_{r})=\min \{n\in \mathbb{N}^{+}:T^{n}(x)\in S_{r}\}.
\end{equation*}

Let us suppose that target sets are of the form $S_{r}=\{x\in X~,~f(x)\leq
r\}$ where $f:X\rightarrow \mathbb{R}^{+}$ is a Lipschitz function and let
us consider the limits 
\begin{equation}
\overline{d}(f)=\underset{r\rightarrow 0}{\lim \sup }\frac{\log \mu (S_{r})}{%
\log (r)}~,~\underline{d}(f)=\underset{r\rightarrow 0}{\lim \inf }\frac{\log
\mu (S_{r})}{\log (r)}.
\end{equation}%
These limits represent a sort of local dimension (the formula for the local
dimension of $\mu $ at a point $x_{0}$ is obtained when $f(x)=d(x,x_{0})$).
When the above limits coincide, let us set $d(f)=\underline{d}(f)=\overline{d%
}(f).$ With this setting, in \cite{Ga}, the following result is proved:

\begin{proposition}
\label{maine}Let $f$ and $S_{r}$ be as above. Then for a.e. $x$%
\begin{equation}
\lim \sup_{r\rightarrow 0}\frac{\log \tau _{T}(x,S_{r})}{-\log r}\geq 
\overline{d}(f)~,~\lim \inf_{r\rightarrow 0}\frac{\log \tau _{T}(x,S_{r})}{%
-\log r}\geq \underline{d}(f).  \label{1211}
\end{equation}%
Moreover, if the system has superpolynomial decay of correlations, as above,
and $d(f)$ exists, then for a.e. $x$ it holds 
\begin{equation}
\lim_{r\rightarrow 0}\frac{\log \tau _{T}(x,S_{r})}{-\log r}=d(f).
\label{lla}
\end{equation}
\end{proposition}

Now, we are ready to prove Theorem \ref{GeneralResult}.

\begin{proof}
(of Theorem \ref{GeneralResult}) Let us suppose $\epsilon $ so small that $d$
is near to 
\begin{equation*}
\lim_{l\rightarrow 0}\frac{\log \mu (C_{\epsilon ,l})}{\log l}.
\end{equation*}%
First notice that we can suppose that $\mu (C_{\epsilon ,l})\rightarrow 0$
(as $l\rightarrow 0$), if it not so, passing to the limit, we have $d=0$. On
the other side, since the system is ergodic $\tau (x,B_{l})$ is bounded,
giving also 
\begin{equation*}
\lim_{l\rightarrow 0}\frac{\log \tau (x,B_{l})}{-\log l}=0.
\end{equation*}

Now, for $\epsilon \geq 0$, let us denote by $C_{-\epsilon ,l}$ the set $%
\{x\in X,s.t.~\Phi ^{t}(x)\in B_{l},0\geq t>-\epsilon \}$.

Consider $D=C_{2\epsilon ,l}\cup C_{-\epsilon ,l}$ and define a projection $%
\text{pr}:D\rightarrow B_{l}$ induced by the flow. Let us consider the time
needed for a point of $D$ to reach $B_{l}$, following the flow or its inverse%
\begin{equation*}
\overline{t}(x)=\left\{ 
\begin{array}{c}
\min \{t\geq 0|\Phi ^{t}(x)\in V\}~\text{ if }~x\in C_{2\epsilon ,l} \\ 
\max \{t\leq 0|\Phi ^{t}(x)\in V\}~\text{ if }~x\in C_{-\epsilon ,l}%
\end{array}%
\right.
\end{equation*}%
($\overline{t}:D\rightarrow \mathbb{R}$). Let us now define the projection pr%
$(x)=\Phi ^{\overline{t}(x)}(x)$.

Since the flow is transverse to $V$ and $\overline{B_{l}}$ is compact for $l$
small enough, we can choose $\epsilon $ and $\overline{l}$ so small that we
can apply the flow box theorem inside $C_{\epsilon ,\overline{l}}$ and pr is
Lipshitz map.

Now, to apply Proposition \ref{maine} we need to define a Lipschitz function 
$\overline{F}$ such that the sets $C_{\epsilon ,l}$ are contained in its
sublevels. To do this let us first define $h_{\epsilon }:D\rightarrow 
\mathbb{R}^{+}$ by $h_{\epsilon }(x)=\varphi _{\epsilon }(\overline{t}(x))$,
where 
\begin{equation*}
\varphi _{\epsilon }(t)=\left\{ 
\begin{array}{c}
-\overline{l}/\epsilon \cdot t~\text{ if }~t<0 \\ 
0~\text{ if }~0\leq t\leq \epsilon \\ 
\overline{l}/\epsilon \cdot (t-\epsilon )~t\geq \epsilon%
\end{array}%
\right. .
\end{equation*}

Finally define $\overline{F}:M\rightarrow \mathbb{R}^{+}$ by%
\begin{equation*}
\overline{F}(x)=\left\{ 
\begin{array}{c}
\overline{l},~\text{ if }~x\notin D \\ 
\max \left( f(\text{pr}(x)),h_{\epsilon }(x)\right) ~\text{ if }~x\in D%
\end{array}%
\right. .
\end{equation*}

Since $\text{pr}$ and all the other involved functions are Lipschitz then
also $\overline{F}$ is. Fixing $\epsilon $, as $l$ varies the sets $%
C_{\epsilon ,l}$ are contained in the sublevels $S_{l}=\overline{F}%
^{-1}([0,l])$. Notice that $\mu (C_{\epsilon ,l})\leq \mu (S_{l})\leq 3\mu
(C_{\epsilon ,l})$. We can consider the discrete time system induced by the
time-$\epsilon $ map of the flow. Let $\tau _{\epsilon }(x,B)=\min
\{n>0,\Phi ^{n\epsilon }(x)\in B\}$ be the hitting time for the time-$%
\epsilon $ map induced by the flow. Applying Proposition \ref{maine}
(Equation \ref{lla}) to these sets we obtain that for $\mu $-almost each $x$%
\begin{equation}
\lim_{l\rightarrow 0}\frac{\log \tau _{\epsilon }(x,S_{l})}{-\log l}%
=\lim_{l\rightarrow 0}\frac{\log (\mu (S_{l}))}{\log l}=\lim_{l\rightarrow 0}%
\frac{\log \mu (C_{\epsilon ,l})}{\log l}.  \label{111}
\end{equation}%
Let us consider $S=\cap _{l}S_{l}$, in a way similar as before, remark that
if $\mu (S)>0$ then $\lim_{l\rightarrow 0}\frac{\log (\mu (S_{l}))}{\log l}%
=0 $. Moreover, since the system is ergodic, $\tau (x,B_{l})$ is eventually
constant and also 
\begin{equation*}
\frac{\log \tau (x,B_{l})}{-\log l}\rightarrow 0
\end{equation*}%
a.e. and the wanted result is proved.

Hence it remains to consider the case $\mu (S)=0$. In this case we can
suppose $x\notin S_{l}$ for some $l$ small enough, which is a full measure
condition.

Considering that if $\Phi ^{\epsilon \tau _{\epsilon }(x,S_{l})}(x)\in
C_{-\epsilon ,l}$ and $x\notin S_{l}$ then the flow starting from $x$ has
already hit $B_{l}$ before arriving to $\Phi ^{\epsilon \tau _{\epsilon
}(x,S_{l})}(x)$ we have that 
\begin{equation}
\tau (x,B_{l})\leq \epsilon \tau _{\epsilon }(x,S_{l})+2\epsilon .
\label{1x}
\end{equation}

On the other hand, since $\Phi ^{\epsilon \left\lfloor \frac{\tau (x,B_{l})}{%
\epsilon }\right\rfloor }(x)\in C_{\epsilon ,l}\subset S_{l}$ then%
\begin{equation}
\tau _{\epsilon }(x,S_{l})\leq \left\lfloor \frac{\tau (x,B_{l})}{\epsilon }%
\right\rfloor  \label{2}
\end{equation}%
$\ $(where $\left\lfloor ~\right\rfloor $ stands for the integer part), and 
\begin{equation}
\tau _{\epsilon }(x,S_{l})\leq \frac{\tau (x,B_{l})}{\epsilon }+1.
\label{33}
\end{equation}%
Supposing $\Phi ^{t}x\notin \cap _{l}C_{\epsilon ,l}$ for each $t\geq 0$,
which is a full measure condition (recall that we suppose $\mu (C_{\epsilon
,l})\rightarrow 0$ ) we then have by $($\ref{1x}$)$ and $($\ref{33}$)$.%
\begin{equation*}
\lim_{l\rightarrow 0}\frac{\log \tau _{\epsilon }(x,S_{l})}{-\log l}%
=\lim_{l\rightarrow 0}\frac{\log \tau (x,B_{l})}{-\log l}
\end{equation*}%
for $\mu -$almost each $x$, and together with $($\ref{111}$)$ we have the
desired result.
\end{proof}

\begin{remark}
We remark that the boundedness and the transversality assumptions on $B_{l}$%
, are used in the proof only to obtain that pr is a Lipschitz map. Hence, if
this can be obtained elsewhere, the result still holds.
\end{remark}

\begin{remark}
By the same proof, using Equation \eqref{1211} it easily follows that in
general measure preserving flows%
\begin{eqnarray}
\underset{l\rightarrow 0}{\lim \sup }\frac{\log \tau (x,B_{l})}{-\log l}
&\geq &\lim_{\epsilon \rightarrow 0}\underset{l\rightarrow 0}{\lim \sup }%
\frac{\log \mu (C_{\epsilon ,l})}{\log l}, \\
\underset{l\rightarrow 0}{\lim \inf }\frac{\log \tau (x,B_{l})}{-\log l}
&\geq &\lim_{\epsilon \rightarrow 0}\underset{l\rightarrow 0}{\lim \inf }%
\frac{\log \mu (C_{\epsilon ,l})}{\log l}.
\end{eqnarray}
\end{remark}

We will prove now a complementary result that deals with Poincar\'{e} maps
of a flow. With this result we can obtain the logarithm law for the flow,
provided we establish it for the induced Poincar\'{e} map on a section and
the return time is integrable. This will not be used in what follows, the
uninterested reader can jump to the next section.

Let $X$ be a metric space, and $\Phi ^{t}$ be a measure preserving flow and $%
\Sigma $ be a section of $(X,\Phi ^{t})$ ; we will show that if the flow is
ergodic and the return time is integrable, then the hitting time scaling
behavior of the flow can be estimated by the one of the system induced on
the section. Hence we can have a logarithm law for the flow if we can prove
it on the section (with the induced return map).

Given any $x\in X$ let us denote with $t(x)$ the smallest strictly positive
time such that $\Phi ^{t(x)}(x)\in \Sigma $. Let also consider $t^{\prime
}(x)$, the smallest non negative time such that $\Phi ^{t^{\prime
}(x)}(x)\in \Sigma $. We remark that these two times differ on the section $%
\Sigma $, where $t^{\prime }=0$ while $t$ is the return time to the section.
We define $\pi :X\rightarrow \Sigma $ as $\pi (x)=\Phi ^{t^{\prime }(x)}(x)$%
, the projection on $\Sigma $. Let us also denote by $\mu _{F}$ the
invariant measure for the Poincar\'{e} map $F$ which is induced by the
invariant measure $\mu $ of the flow.

\begin{proposition}
\label{propsec}Let us suppose that the flow $\Phi ^{t}$ is ergodic and has a
section $\Sigma $ with an induced map $F$ and invariant measure $\mu _{F}$
such that $\int_{\Sigma }t(x)~d\mu _{F}<\infty $. Let $r\geq 0$ and $%
S_{r}\subseteq \Sigma $ be a decreasing family of measurable subsets with $%
\lim_{r\rightarrow 0}\mu _{F}(S_{r})=0$. Let us consider the hitting time
relative to the Poincar\'{e} map 
\begin{equation}
\tau ^{\Sigma }(x,S_{r})=\min \{n\in \mathbb{N}^{+};F^{n}(x)\in S_{r}\}.
\end{equation}

There is a full measure set $C\subseteq X$ such that if $x\in C$%
\begin{equation}
\lim_{r\rightarrow 0}\frac{\log \tau (x,S_{r})}{-\log r}=\lim_{r\rightarrow
0}\frac{\log \tau ^{\Sigma }(\pi (x),S_{r})}{-\log r}.
\end{equation}
\end{proposition}

\begin{proof}
Since $(X,\Phi ^{t},\mu )$ is ergodic, then $(\Sigma ,F,\mu _{F})$ is
ergodic. Let $S_0=\bigcap_{r>0}S_r$, by hypothesis $\mu _{F}(S_{0})=0$ and
the set $A^{\prime }=\Sigma -\cup _{i\geq 0}F^{-i}(S_{0})$ has full measure.
Let us consider the set $A^{\prime \prime }\subset \Sigma $ where $\forall
x\in A^{\prime \prime }$ 
\begin{equation}
\frac{1}{n}\sum_{i=0}^{n}t(F^{i}(x))\longrightarrow \int_{\Sigma }t(x)d\mu
_{F}.
\end{equation}%
By the pointwise ergodic theorem this set has full measure. Now let us set $%
A=A^{\prime }\cap A^{\prime \prime }$.

Let us first assume that $x$ $\in $ $A\subseteq \Sigma $. Then $\tau
^{\Sigma }(x,S_{r})$ and $\tau (x,S_{r})$ are related by 
\begin{equation}
\tau (x,S_{r})=\sum_{i=0}^{\tau ^{\Sigma }(x,S_{r})}t(F^{i}(x)).
\label{sum1}
\end{equation}%
If $x\in A$ (recalling that since $x\in A^{\prime }$ then $\tau ^{\Sigma
}(x,S_{r})\rightarrow \infty $ as $r\rightarrow 0$) then 
\begin{equation*}
\ \frac{1}{\tau ^{\Sigma }(x,S_{r})}\sum_{i=0}^{\tau _{r}^{\Sigma
}(x,S_{r})}t(F^{i}(x))\longrightarrow \int_{\Sigma }t(x)~d\mu _{F}.
\end{equation*}%
Thus%
\begin{equation}
\tau (x,S_{r})=\sum_{i=0}^{\tau _{r}^{\Sigma
}(x,S_{r})}t(F^{i}(x))=c(x,r)\cdot \tau ^{\Sigma }(x,S_{r})\cdot
\int_{\Sigma }t(x)~d\mu _{F}
\end{equation}%
with $c(x,r)\rightarrow 1$ as $r\rightarrow 0$.\ The same is also true for
each $x\in \pi ^{-1}(A)$ which is a full measure set. Extracting logarithms
and taking the limits, we get the required result.
\end{proof}

By Proposition \ref{maine}, in the case where the induced map is
superpolynomially mixing this implies the following

\begin{corollary}
Under the above assumptions suppose that the induced system on $\Sigma $ has
superpolynomial decay of correlations with respect to Lipschitz observables.
Let $f:\Sigma \rightarrow \mathbb{R}^{+}$ be a Lipschitz function and let%
\begin{equation}
B_{l}=\{x\in \Sigma ,f(x)<l\}.
\end{equation}%
Suppose that the following limit exists%
\begin{equation}
d=\lim_{l\rightarrow 0}\frac{\log \mu (B_{l})}{\log l}
\end{equation}%
($\mu $ is the invariant measure induced on the section) then there is a
full measure set $C\subseteq X$ such that if $x\in C$.%
\begin{equation}
\lim_{l\rightarrow 0}\frac{\log \tau (x,B_{l})}{-\log l}=d.
\end{equation}
\end{corollary}

\section{Decay of correlations for contact Anosov flows}

We recall some basic definition and some properties of contact Anosov flow.
These will be used when dealing with the geodesic flow of a negatively
curved manifold.

\begin{definition}
Let $M$ be a $2n+1$-dimensional manifold. If $M$ admits a $1$-form $\alpha$
such that $\alpha\wedge(d\alpha)^n\neq 0$, then $M$ is said to be a \textbf{%
contact manifold}. The hyperplane distribution $\ker\alpha$ is called the 
\textbf{contact distribution}.
\end{definition}

\begin{definition}
Given a flow $\Phi ^{t}$ on $(M,\alpha )$, we call it a \textbf{contact flow}
if it preserves the contact form, i.e. $\alpha(d\Phi ^{t} v) =\alpha(v)$.
Moreover if its infinitesimal generator $V$ is such that $\alpha (V)=1$ and $%
d\alpha (V,v)=0$ for every vector field $v$ we will call it a \textbf{Reeb
vector field}.
\end{definition}

The natural invariant measure for the flow is nothing else than the measure
induced by the volume form $\alpha \wedge (d\alpha )^{n}$. We will denote
this measure as $\mu $. We refer to \cite{Geiges} for the basic facts about
contact manifolds (and much more). A Reeb vector field, by definition, is
transverse to the contact distribution, i.e. $\ker(\alpha_x)\oplus V(x)= T_x
M$ for every $x$ in $M$.

\begin{definition}
A flow $\Phi $ is said to be \textbf{Anosov}, if at each point $x$ the
tangent space $T_{x}M$ \ can be split into a direct sum $T_{x}M=E_{x}^{u}%
\oplus E_{x}^{0}\oplus E_{x}^{s}$ respectively called unstable, central and
stable directions such that the splitting is invariant with respect to the
action of the flow, $E_{x}^{0}$ is one dimensional and coincides with the
direction of the flow and there are $A,B>0$ such that 
\begin{eqnarray*}
||d\Phi ^{t}(v)|| &\leq &Ae^{-Bt}||v||~for~each~v\in E_{x}^{s}~and~t\geq 0 \\
||d\Phi ^{t}(v)|| &\leq &Ae^{-Bt}||v||~for~each~v\in E_{x}^{u}~and~t\leq 0
\end{eqnarray*}
\end{definition}

In \cite{Li} the following nontrivial fact is proved.

\begin{theorem}[Exponential decay of correlations for contact Anosov flows]
\label{Liverani} Let $M$ be a $C^{4}$, $2d+1$ connected compact manifold and
let us consider be a $C^{4}$ flow $\Phi ^{t}:M\rightarrow M$ defined on it.
Suppose this flow is Anosov and that it preserves a contact form. Then,
there exist constants $C_{1},C_{2}$ such that for each $\alpha $-H\"{o}lder
functions $\varphi ,\psi $ on $M$, 
\begin{equation*}
\bigg|\int \phi \circ \Phi ^{t}\psi d\mu -\int \phi d\mu \int \psi d\mu %
\bigg|\leq C_{1}\left\vert \left\vert \phi \right\vert \right\vert _{\alpha
}\left\vert \left\vert \psi \right\vert \right\vert _{\alpha }e^{-C_{2}t}.
\end{equation*}
\end{theorem}

In the above equation $\left\vert \left\vert .\right\vert \right\vert
_{\alpha }$ represents the H\"{o}lder norm. Given a Riemannian manifold, the
geodesic flow on the unit tangent bundle is a Contact Flow and its
infinitesimal generator is a Reeb vector field; moreover, such a contact
flow is Anosov in the case of Riemannian manifolds with negative sectional
curvature (see \cite{KatokHasselblat}, ch 17.6 ). Altogether, this simply
means that the following holds.

\begin{corollary}
\label{CurvaturaNegativa} The geodesic flow $\Phi ^{t}$ of a $C^{4}$
compact, connected manifold with strictly negative sectional curvature is
exponentially mixing with respect to H\"{o}lder observables as above.
\end{corollary}

\section{Logarithm law for the geodesic flow}

Let $M$ be an hyperbolic manifold (with constant negative curvature) of
dimension $n$ and $T^{1}M$ be its unitary tangent bundle. Let $\pi
_{1}:T^{1}M\rightarrow M$ be the canonical projection, $\Phi ^{t}$ be the
geodesic flow on $T^{1}M$, $\mu $ the Liouville measure on $T^{1}M$ , and $d$
the Riemannian distance on $M.$ In \cite{Maucourant} the following result is
proved.

\begin{theorem}[Maucourant]
\label{mau}For all $p\in M$ and $\mu $ almost each $v\in T^{1}M$%
\begin{equation}
\underset{t\rightarrow \infty }{\lim \sup }\frac{-\log d(p,\pi (\Phi ^{t}v))%
}{\log t}=\frac{1}{n-1}.  \label{maueq}
\end{equation}
\end{theorem}

We now apply Theorem \ref{GeneralResult} to obtain a version of Theorem \ref%
{mau} for compact manifolds with \emph{variable }negative sectional
curvature.

We remark that results of this kind are also stated in $CAT(-1)$ spaces, and
proved with completely different techniques (\cite{HP}).

To apply Theorem \ref{GeneralResult}, we have to find suitable transversal
sections to the flow, which can take account of the intersection of a
geodesics with the ball centered at the target point on the base manifold.
We fix now some notation that is used throughout the section; again denote
by $M$ a Riemannian manifold with metric $g$ and negative sectional
curvature, by $\pi :TM\rightarrow M$ its tangent bundle, by $\pi
_{1}:T^{1}M\rightarrow M$ its unitary tangent bundle. A coordinate
neighborhood of $M$, $(U,x^{1},\ldots ,x^{n})$ is centered in a point $p$
when $x^{1}(p)=x^{2}(p)=\ldots =x^{n}(p)=0$, fixed a coordinate neighborhood
centered at $p$ it is understood that we will denote the point $p$ also as
the point $0$. When we deal with a point $q$ on the tangent bundle $TM$, and
we restrict ourselves to a trivializing neighborhood, we shall use without
any remark the notation $(x,v)$ to denote its coordinates in the local
trivialization. Another assumption we make is that, when we choose a
trivializing neighborhood for $TM$, we choose the coordinate system so that
the contact form $\alpha $ is written in coordinates as $%
\sum_{i,j=1}^{n}g_{ij}(x)v^{i}dx^{j}$. The existence of such a coordinate
system is a well known result and can be found e.g. in \cite{KatokHasselblat}%
.

The following estimation of the hitting time of the geodesic flow will imply
an equation similar to $($\ref{maueq}$)$ (and some even more precise result).

\begin{theorem}
\label{loglaw3}Let $\Phi ^{t}$ be the geodesic flow on the unitary tangent
bundle of $M$, $n$-dimensional compact, connected manifold with negative
sectional curvature. Let us fix $p\in M$ and let $f:T^{1}M\rightarrow 
\mathbb{R}$ be given by $f(q)=d(\pi _{1}(q),p)$, where $d$ is the Riemannian
distance on $M$. Let us denote 
\begin{equation*}
U_{r}(p)=\{x\in T^{1}M,f(x)<r\}
\end{equation*}%
and by $\tau (x,U_{r}(p))$ as before, the time needed for a point $x$ in $%
T^{1}M$ to reach $U_{r}(p)$ under the action of the flow. We have that for
almost every point $x$ in $T^{1}M$ 
\begin{equation*}
\lim_{r\rightarrow 0}\frac{\log (\tau (x,U_{r}(p)))}{-\log (r)}=n-1.
\end{equation*}
\end{theorem}

\begin{remark}
Please remark that $\pi _{1}(U_{r}(p))$ is the geodesic ball centered at $p$
in $M$; therefore, given a point $x$ in $T^{1}M$ we have that $\tau
(x,U_{r}(p))$ coincides with the time which is needed for a geodesic on $M$
having initial data $x$ to enter in the ball $B_{r}(p)$.
\end{remark}

To prove the theorem we have to find a suitable section where to apply
Theorem \ref{GeneralResult}. In the following lemma we define it and prove
that it is a submanifold of $T^{1}M$.

\begin{lemma}
\label{Lemma:Transverse} Let $(U,x_{1},\ldots ,x_{n})$ be a coordinate
neighborhood centered in $p$, such neighborhood is a trivializing
neighborhood for $\pi :TM\rightarrow M$. We denote by $(x^{1},\ldots
,x^{n},v^{1},\ldots ,v^{n})$ the coordinates on $\pi ^{-1}(U)$. Let $h$ be
the function from $\pi ^{-1}(U)\rightarrow \mathbb{R}$ given in coordinates
by 
\begin{equation*}
h(x,v)=\sum_{i,j=1}^{n}g_{ij}(x)x^{i}v^{j},
\end{equation*}%
where the $g_{ij}$ are the coefficients of the Riemannian metric on $TM$.
Then, there exists an open subset of $U$, that, by abuse of notation, we
shall denote again by $U$, such that 
\begin{equation*}
T(p):=h^{-1}(0)\cap T^{1}M\cap \pi ^{-1}(U)
\end{equation*}%
is a submanifold of $T^{1}M$.
\end{lemma}

\begin{proof}
The proof of this fact is an application of the submersion theorem and we
report it for completeness. We have that 
\begin{equation*}
T^{1}M\cap \pi ^{-1}(U)=\{(x^{1},\ldots ,x^{n},v^{1},\ldots ,v^{n})\mid
\left\vert \left\vert v\right\vert \right\vert =1\}.
\end{equation*}%
First of all we remark that $h^{-1}(0)$ is not a submanifold of $\pi
^{-1}(U) $ (inside $TM$). Indeed, we compute the differential of $h$: 
\begin{equation*}
dh=\sum_{k=1}^{n}\bigg(\sum_{i,j=1}^{n}\frac{\partial g_{ij}}{\partial x^{k}}%
x^{i}v^{j}+g_{kj}v^{j}\bigg)dx^{k}+\sum_{i,j=1}^{n}g_{ij}x^{i}dv^{j};
\end{equation*}%
as it is, this differential is not surjective for every point of $h^{-1}(0)$
and we cannot apply the submersion theorem; the counterimage of $0$ through $%
h$ is not a submanifold. To see this, simply remark that $dh_{(0,0)}=0$. Now
we think of $T^{1}M$ as the submanifold of $TM$ given as the counterimage of 
$n^{-1}(1)$ where $n$ is given in coordinates as $n(x,v)=%
\sum_{i,j=1}^{n}g_{ij}(x)v^{i}v^{j}$. What we want to prove is that for
every point $(0,v)$ in $\pi ^{-1}(0)$ we have that $\ker (dh_{(0,v)})$ is
not contained in $T_{(0,v)}T^{1}M$; therefore, for every point in $\pi
_{1}^{-1}(0)$, $dh$ is surjective and we can apply the submersion theorem.
If we restrict ourselves to $\pi _{1}^{-1}(0)$, since all the $x^{i}$ are $0$%
, we have that $dh$ takes the form: 
\begin{equation*}
dh\bigg|_{\pi _{1}^{-1}(0)}=\sum_{i,j=1}^{n}g_{ij}(0)v^{j}dx^{i};
\end{equation*}%
We compute also: 
\begin{equation*}
dn_{(0,v)}=\sum_{k=1}^{n}\sum_{i,j=1}^{n}\frac{\partial g_{ij}}{\partial
x^{k}}(0)v^{i}v^{j}dx^{k}+\sum_{i,j=1}^{n}g_{ij}(0)v^{i}dv^{j}.
\end{equation*}%
We prove the claim showing that for each point $(0,v)$ we have that $%
dn_{(0,v)}$ and $dh_{(0,v)}$ are linearly independent; to show this we make
some assumptions. As we said already, the coordinate chart is centered in $p$
and we are working on $\pi _{1}^{-1}(0)$; moreover we can take a coordinate
system such that $g_{ij}(0)=\delta _{ij}$, i.e. a coordinate system in which
the coordinate fields are orthogonal in the origin (they cannot be
orthogonal in a neighborhood of the origin unless $M$ is flat). In this
coordinate system we have that: 
\begin{equation*}
dh_{(0,v)}=\sum_{k=1}^{n}v^{k}dx^{k}
\end{equation*}%
and 
\begin{equation*}
dn_{(0,v)}=\sum_{k=1}^{n}\sum_{i,j=1}^{n}\bigg(\frac{\partial g_{ij}}{%
\partial x^{k}}v^{i}v^{j}\bigg)dx^{k}+\sum_{i=1}^{n}v^{i}dv^{i}.
\end{equation*}%
To prove this two sections are linearly independent is the same as proving
that the following matrix has rank $2$: 
\begin{equation*}
\left( 
\begin{array}{cccccc}
\sum_{i,j=1}^{n}\frac{\partial g_{ij}}{\partial x^{1}}v^{i}v^{j} & \cdots & 
\sum_{i,j=1}^{n}\frac{\partial g_{ij}}{\partial x^{n}}v^{i}v^{j} & v^{1} & 
\cdots & v^{n} \\[0.3em] 
v^{1} & \cdots & v^{n} & 0 & \cdots & 0 \\ 
&  &  &  &  & 
\end{array}%
\right) .
\end{equation*}%
Since $T^{1}M_{0}=\{v\in TM_{0}\mid \sum_{i=1}^{n}(v^{i})^{2}=1\}$ we have
that the determinant $(v^{k})^{2}$ of at least one of the $2\times 2$ minors
of the form: 
\begin{equation*}
\left( 
\begin{array}{cc}
\sum_{i,j=1}^{n}\frac{\partial g_{ij}}{\partial x^{k}}v^{i}v^{j} & v^{k} \\%
[0.3em] 
v^{k} & 0 \\ 
& 
\end{array}%
\right)
\end{equation*}%
is different from $0$. From the submersion theorem we have that for each
point $(0,v)$ in $\pi _{1}^{-1}(0)$ there exists an open set $V_{(0,v)}$ in $%
T^{1}M|_{U}$ on which $h$ is a submersion. Since $\pi _{1}^{-1}(0)$ is
compact, we can find a finite collection of open sets $V_{(0,v_{i})}$ whose
union is a neighborhood $V$ of $\pi _{1}^{-1}(0)$ and where $h$ is a
submersion. Therefore, we can rescale $U$ so that $h^{-1}(0)\cap T^{1}M\cap
\pi ^{-1}(U)$ is a well defined submanifold of $T^{1}M$.
\end{proof}

\begin{lemma}
In the hypothesis of Lemma \ref{Lemma:Transverse}, there exists an $\bar{l}$
such that for every $l<\bar{l}$ we have that 
\begin{equation*}
T_{l}(p)=\{x\in T(p),f(x)<l\}
\end{equation*}
is transverse to the flow.
\end{lemma}

\begin{proof}
Consider the submanifold $T(p)\subset T^{1}M$. For each point in $\pi
_{1}^{-1}(0)$ (and only for these points) we have that the tangent space to $%
T(p)$ is given by the kernel of $dh|_{(0,v)}\equiv \alpha _{(0,v)}$, where $%
\alpha$ is the contact form; therefore, $T(p)$ is transversal to the
geodesic flow on $\pi _{1}^{-1}(0)$. For each point $q$ in $\pi_1^{-1}(0) $
we can take an open neighborhood $V_q$ where the flow is transverse to $T(p)$%
; since $\pi_1^{-1}(0)$ is compact we can extract a finite cover $V_{q_i}$;
we take now 
\begin{equation*}
\bar{l}=\min_{i}\sup_{\tilde{q}\in U_{q_i}}f(\tilde{q});
\end{equation*}
for each $l<\bar{l}$ we have that $T_l(p)$ is transverse to the geodesic
flow.
\end{proof}

\begin{remark}
Please remark that for $l$ small enough $\pi_1(T_l(p))$ is a geodesic ball
of radius $l$. Indeed if we take a geodesic ball $B_{\tilde{l}}(p)$ of
radius $\tilde{l}$ and fix an $x$ in $B_{l}(0)$: we can solve the linear
homogeneous equation $\sum_{i=1}^{n}g_{ij}(x)x^{i}v^{j}=0$ in the unknown $v$%
; this is equivalent to find a vector in the orthogonal complement to the
vector $w^j=(g_{ij}(x)x^{i})^j$ in $\mathbb{R}^{2n}$ with the canonical
scalar product, since the canonical scalar product is nondegenerate such a
vector exists. If we chose the radius $\tilde{l}$ small enough the couple $%
(x,v)$ belongs to $T_{\tilde{l}}(p)$.
\end{remark}

We can now prove Theorem \ref{loglaw3}.

\begin{proof}
(of Theorem \ref{loglaw3}) We found a suitable transverse section to the
flow and we know that the flow is superpolynomially mixing (see Corollary %
\ref{CurvaturaNegativa}) so we are in the condition to apply Theorem \ref%
{GeneralResult}. To this extent we have to evaluate the right hand of (\ref%
{general}).

The invariant volume form $\omega=\alpha \wedge (d\alpha )^{n-1}$ of $T^{1}M$
takes the form: 
\begin{equation*}
\omega=\alpha \wedge (d\alpha )^{n-1}=\det
(g_{hk}(x))\sum_{i=1}^{n}(-1)^{i}v^{i}dv^{1}\wedge \cdots \wedge \widehat{%
dv^{i}}\wedge \cdots \wedge dv^{n}\wedge dx^{1}\wedge \cdots \wedge dx^{n},
\end{equation*}%
where $\det (g_{hk}(x))$ is the determinant of the matrix of the metric.

We cover $T_{l}(p)$ by open sets 
\begin{equation*}
V_i=\{(x^1,\ldots,x^n,v^1,\ldots,v^n)\in T_l(p)\mid \sum_{k=1}^n g_{ik}
v^k\neq 0\},
\end{equation*}
on each $V_i$ a point $(x,v)$ of $T_{l}(p)$ has $i$-th coordinate 
\begin{equation*}
x^i(x^1,\ldots,\widehat{x^i},\ldots,x^n,v^1,\ldots,v^n)=\frac{%
-\sum_{j,k=1;j\neq i}^n g_{jk}x^jv^k}{\sum_{k=1}^n g_{ik} v^k},
\end{equation*}
where the hat means we are not taking the coordinate into consideration. We
can give also explictly a partition of unity $\{\phi_i\}$ subordinate to the
cover $V_i$, by defining 
\begin{equation*}
\phi_i:= v^i \sum_{j=1}^n g_{ij} v^j= \sum_{j=1}^n g_{ij} v^i v^j;
\end{equation*}
clearly $\sum_{i=1}^n \phi_i=1$.

We can apply the flow box theorem to the geodesic flow and build an atlas
for $C_{\epsilon ,l}$ where the coordinate neighborhoods are given by 
\begin{equation*}
W_{i}:=\{\Phi ^{-t}(x)\mid x\in V_{i},t\in \lbrack 0,\epsilon )\},
\end{equation*}%
with partition of unity $\{\psi _{i}\}$, the pull-back through the flow of
the partition of unity $\{\phi _{i}\}$. By definition 
\begin{equation*}
\int_{C_{\epsilon ,l}}\omega =\sum_{i}\int_{W_{i}}\psi _{i}\cdot \omega .
\end{equation*}%
The coordinates for $C_{\epsilon ,l}$ are induced on each $W_{i}$ by the
geodesic flow: 
\begin{align*}
(x_{1},\ldots ,\widehat{x^{i}},\ldots ,x^{n},& v^{1},\ldots ,v^{n},t)\mapsto 
\\
& \Phi ^{-t}(x^{1},\ldots ,x^{i}(x^{1},\ldots ,\widehat{x^{i}},\ldots
,x^{n},v^{1},\ldots ,v^{n}),\ldots ,x^{n},v^{1},\ldots ,v^{n}).
\end{align*}%
Through a linear change of coordinates, we can choose coordinates such that
the matrix of the metric is such that, over the point $p$, it is represented
in coordinates by the identity matrix; as remarked above we can do this in
the fiber over the point $p$ but not in a neighborhood of this fiber since
this is equivalent to flatness.

From basic properties of the geodesic flows it follows that, in we take any
point $q=(0,w)$ in the fiber over $p$, we have that, for every $k=1,\ldots ,n
$ 
\begin{equation*}
\frac{\partial x^{k}}{\partial t}(q)=w^{k}.
\end{equation*}%
Moreover, if $j=1,\ldots ,i-1,i+1,\ldots ,n$, $k=1,\ldots ,n$, $l=1,\ldots ,n
$ 
\begin{equation*}
\frac{\partial x^{j}}{\partial x^{k}}(q)=\delta _{k}^{j},\quad \frac{%
\partial v^{l}}{\partial x^{k}}(q)=0,\quad \frac{\partial v^{l}}{\partial
v^{k}}(q)=\delta _{k}^{l},\quad \frac{\partial v^{l}}{\partial x^{k}}(q)=0,
\end{equation*}%
where $\delta _{k}^{l}=1$ if $k=l$ and $0$ otherwise. If $j=i$ 
\begin{equation*}
\frac{\partial x^{i}}{\partial x^{j}}(q)=\frac{w^{j}}{w^{i}}.
\end{equation*}%
If we compute the determinant of the $n\times n$ matrix given by 
\begin{equation*}
\left[ 
\begin{array}{c|c}
\partial x^{k}/\partial x^{j}(q) & \partial x^{k}/\partial t(q)%
\end{array}%
\right] 
\end{equation*}%
we see it is $1/w^{i}$. This implies that, on $W_{i}$, if $\epsilon $ and $l$
are small enough, the pull-back of the volume form through this coordinate
change, near the fiber over $p$, can be written as 
\begin{align*}
\frac{1}{v^{i}}\sum_{j=1}^{n}(-1)^{j}v^{j}dv^{1}\wedge \cdots \wedge 
\widehat{dv^{j}}\wedge \cdots \wedge dv^{n}& \wedge dx^{1}\wedge \ldots
\wedge \widehat{dx^{i}}\wedge \ldots \wedge dx^{n}\wedge dt \\
& +\omega ^{\prime }(\epsilon ,l)
\end{align*}%
where $\omega ^{\prime }(\epsilon ,l)$ is a $2n-1$ form tending to $0$ as $%
\epsilon ,l\rightarrow 0$. This implies that the integral 
\begin{equation*}
\int_{W_{i}}\psi _{i}\cdot \omega =K_{1}\cdot \epsilon \cdot l^{n-1}+\text{%
higher order terms in $\epsilon $ and $l$},
\end{equation*}%
where $K_{1}$ is a constant. This tells us that 
\begin{equation*}
\mu (C_{\epsilon ,l})=K_{2}\cdot \epsilon \cdot l^{n-1}+\text{higher order
terms in $\epsilon $ and $l$},
\end{equation*}%
where $K_{2}$ is a constant. Therefore 
\begin{equation*}
\lim_{\epsilon \rightarrow 0}\lim_{l\rightarrow 0}\frac{\log \mu
(C_{\epsilon ,l})}{\log l}=n-1.
\end{equation*}%
Then Theorem \ref{GeneralResult} implies that 
\begin{equation*}
\lim_{r\rightarrow 0}\frac{\log \tau (x,T_{r}(p))}{-\log r}=\lim_{\epsilon
\rightarrow 0}\lim_{l\rightarrow 0}\frac{\log \mu (C_{\epsilon ,l})}{\log l}%
=n-1.
\end{equation*}%
The geodesics are parametrized by arclength and therefore the time of the
flow coincides with the length of the geodesic. Let us take a typical point $%
x$ in $T^{1}M$. We have that for $r$ small enough 
\begin{equation*}
\tau (x,T_{2r}(p))-r\leq \tau (x,U_{r}(p))\leq \tau (x,T_{r}(p)).
\end{equation*}%
Consequently: 
\begin{eqnarray*}
\lim_{r\rightarrow 0}\frac{\log (\tau (x,T_{r}(p)))}{-\log (r)}
&=&\lim_{r\rightarrow 0}\frac{\log (\tau (x,T_{2r}(p))-r)}{-\log (2r)} \\
&\leq &\lim_{r\rightarrow 0}\frac{\log (\tau (x,U_{r}(p)))}{-\log (r)}\leq
\lim_{r\rightarrow 0}\frac{\log (\tau (x,T_{r}(p)))}{-\log (r)}.
\end{eqnarray*}%
And the statement is established.
\end{proof}

With similar methods we can also have other results. In next proposition we
look for the behavior of the time needed for a geodesics to reach a certain
point on the base space, arriving from a prescribed direction, as it is
intuitive, in this case the exponent is much bigger.

\begin{proposition}
\label{Theo:Loglaw4}In the same assumptions of Theorem \ref{loglaw3}, let us
fix $q\in T^{1}M$ and let $d_{S}$ be the Riemannian distance on $T^{1}M$
with respect to the Sasaki metric (the natural metric on the tangent
bundle). Let us denote 
\begin{equation*}
\tilde{B}_{r}(q)=\{x\in T^{1}M,d_{S}(x,q)<r\},
\end{equation*}%
the geodesic ball in $T^{1}M$ with radius $r$, and by $\tau (x,\tilde{B}%
_{r}(q))$ as before, the time needed for a point $x$ in $T^{1}M$ to reach $%
\tilde{B}_{r}(q)$ under the action of the flow. We have that for almost
every point $x$ in $T^{1}M$ 
\begin{equation*}
\lim_{r\rightarrow 0}\frac{\log (\tau (x,\tilde{B}_{r}(q)))}{-\log (r)}=2n-2.
\end{equation*}
\end{proposition}

\begin{proof}
In this case we take any transverse, differentiable section to the geodesic
flow and we take the intersection of $\tilde{B}_{r}(q)$ with this section;
on this target we build the cylinder $\tilde{C}_{\epsilon ,l}$. Following
the line of Theorem \ref{loglaw3}, we see that 
\begin{equation*}
\lim_{\epsilon \rightarrow 0}\lim_{r\rightarrow 0}\frac{\log \mu (\tilde{C}%
_{\epsilon ,r})}{\log r}=2n-2.
\end{equation*}%
Now, Theorem \ref{GeneralResult} implies: 
\begin{equation*}
\lim_{r\rightarrow 0}\frac{\log \tau (x,\tilde{T}_{r}(p))}{-\log r}%
=\lim_{\epsilon \rightarrow 0}\lim_{r\rightarrow 0}\frac{\log \mu (\tilde{C}%
_{\epsilon ,r})}{\log r}=2n-2.
\end{equation*}%
Which easily gives the statement as above.
\end{proof}

Theorem \ref{loglaw3} gives an estimation for the hitting time of particular
sets for the geodesic flow. Now we consider the behavior of the distance
between the orbit of the typical point $x$ and the target point $p$, which
is the "center" of the set. This will give the following generalizations and
extensions of the above cited Theorem \ref{mau}.

\begin{proposition}
\label{loglaw4}Let $M$ be a compact, connected, $C^{4}$ manifold of
dimension $n$ with strictly negative sectional curvature and $T^{1}M$ be its
unitary tangent bundle. Let $\pi _{1}:T^{1}M\rightarrow M$ be the canonical
projection,$\ \mu $ the Liouville measure on $T^{1}M$ , and $d$ the
Riemannian distance on $M,$ then for each $p\in M$:%
\begin{equation}
\underset{t\rightarrow \infty }{\lim \sup }\frac{-\log d(p,\pi _{1}(\Phi
^{t}x))}{\log t}=\frac{1}{n-1}
\end{equation}%
holds for $\mu $ almost each $x\in T^{1}M.$
\end{proposition}

\begin{proof}
Let $d_{t}(x,p)=\inf_{s\leq t}\mathrm{dist}(\pi _{1}(\Phi ^{s}(x)),p)$. By
the ergodicity of the flow $d_{t}\rightarrow 0$ as $t\rightarrow \infty $.
Without loss of generality we can suppose that for $t$ big enough $%
d_{t}=t^{-\alpha (t)}$. By Theorem \ref{loglaw3} 
\begin{equation*}
\underset{t\rightarrow \infty }{\lim }\frac{\log \tau (x,U_{t^{-\alpha (t)}})%
}{-\log (t^{-\alpha (t)})}=n-1,
\end{equation*}%
if $t$ is big enough, $\forall \epsilon >0$, then 
\begin{equation}
t^{\alpha (t)(n-1)-\alpha (t)\epsilon }\leq \tau (x,U_{t^{-\alpha (t)}})\leq
t^{\alpha (t)(n-1)+\alpha (t)\epsilon }
\end{equation}%
but by the way $\alpha $ is defined $\tau (x,U_{t^{-\alpha (t)}})\leq t$;
hence $t^{\alpha (t)(n-1)-\alpha (t)\epsilon }\leq t$ , $\alpha
(t)(n-1)-\alpha (t)\epsilon \leq 1$ and $\alpha (t)\leq \frac{1}{%
n-1-\epsilon }$. This implies that 
\begin{equation}
\underset{t\rightarrow \infty }{\lim \sup }\frac{-\log d_{t}(p,\pi _{1}(\Phi
^{t}x))}{\log t}\leq \frac{1}{n-1}.  \label{logla4eq}
\end{equation}%
On the other hand there are infinitely many $t$ such that $\tau
(x,B_{t^{-\alpha (t)}})=t$. This, with the same computations as above
implies the statement since $d_{t}\leq d$ and $d_{t}=d$ when $\tau
(x,B_{t^{-\alpha (t)}})=t$.
\end{proof}

Our estimation of the hitting time allows also to state a relation which can
be seen as a strong form of the logatithm law given in (\ref{loglaw1}) and (%
\ref{maueq}).

\begin{proposition}
\label{loglaw5}Under the assumptions of the above Proposition \ref{loglaw4}.
Let 
\begin{equation*}
d_{t}(x,p)=\inf_{s\leq t}\mathrm{dist}(\pi _{1}(\Phi ^{s}(x)),p)
\end{equation*}%
then for each $p\in M$:%
\begin{equation}
\underset{t\rightarrow \infty }{\lim }\frac{-\log d_{t}(p,\pi _{1}(\Phi
^{t}x))}{\log t}=\frac{1}{n-1}
\end{equation}%
holds for almost each $x\in T^{1}M$.
\end{proposition}

\begin{proof}
After Proposition \ref{loglaw4} (see Equation \ref{logla4eq}), it is
sufficient to prove 
\begin{equation*}
\underset{t\rightarrow \infty }{\lim \inf }\frac{-\log d_{t}(p,\pi _{1}(\Phi
^{t}x))}{\log t}\geq \frac{1}{n-1}.
\end{equation*}%
Suppose the converse is true: there is a sequence $t_{i}$ such that 
\begin{equation*}
\underset{i\rightarrow \infty }{\lim }\frac{-\log d_{t}(p,\pi _{1}(\Phi
^{t_{i}}x))}{\log t_{i}}<\frac{1}{n-1}.
\end{equation*}%
Then there exist an $\epsilon >0$ such that 
\begin{equation*}
d_{t}(p,\pi _{1}(\Phi ^{t_{i}}x))>t_{i}^{\frac{1}{-n+1-\epsilon }}
\end{equation*}%
and then 
\begin{equation*}
\tau (x,U_{t_{i}^{\frac{1}{-n+1-\epsilon }}})>t_{i}.
\end{equation*}%
Setting $l_{i}=t_{i}^{\frac{1}{-n+1-\epsilon }},$ this gives $\tau
(x,U_{l_{i}})>l_{i}^{-n+1-\epsilon }$ and 
\begin{equation*}
\frac{\log \tau (x,U_{l_{i}})}{-\log l_{i}}>n-1+\epsilon
\end{equation*}%
contradicting 
\begin{equation*}
\underset{t\rightarrow 0}{\lim }\frac{\log \tau (x,U_{t})}{-\log (t)}=n-1,
\end{equation*}%
and therefore proving the theorem.
\end{proof}

\end{document}